\documentclass{amsart}
\usepackage{amssymb}
\usepackage{amsfonts}
\usepackage{amsmath}
\usepackage{url}
\usepackage{color}
\usepackage{setspace}
\usepackage{enumerate}

\usepackage{epsfig}
\usepackage{graphicx}
\usepackage{psfrag}

\newtheorem{theorem}{Theorem}

\newtheorem{proposition}[theorem]{Proposition}
\newtheorem{corollary}[theorem]{Corollary}
\newtheorem{lemma}[theorem]{Lemma}
\newtheorem{claim}[theorem]{Claim}

\newtheorem{question}[theorem]{Question}

\theoremstyle{remark}

\newtheorem{observation}[theorem]{Observation}

\newtheorem{definition}[theorem]{Definition}
\newtheorem{fact}[theorem]{Fact}

\begin{document}

\title{Deforming ideal solid tori}
\author[F. Gu\'eritaud]{Fran\c{c}ois Gu\'eritaud}
\date{May 2009}

\begin{abstract}
We describe the deformation space of a solid torus with boundary modelled on convex ideal hyperbolic polyhedra. This deformation space is given by natural Gauss--Bonnet type inequalities on the dihedral angles. The result extends to solid tori with an arbitrary conical singularity along the core.  Our method is to decompose the torus into tetrahedra and maximize a volume functional.
\end{abstract}

\maketitle

\section{Introduction}

There are several interesting ways to parametrize the moduli space of hyperbolic objects (subsets of $\mathbb{H}^3$ or quotients thereof) by geometric data associated to their boundary. For Kleinian surface--group manifolds $M=\mathbb{H}^3/\pi_1(S)$, this boundary data can be a conformal structure on the union of two copies of the surface $S$, as in the Ahlfors--Bers theorem (then $S\sqcup S \simeq \partial_\infty M$) and its enhancement, the Ending Lamination Classification (where the data can partially or completely ``degenerate'' to projective measured laminations on $S$). For geometrically finite surface--group manifolds $M$, Thurston also proposed his Bending Lamination Conjecture (BLC): this time, the data is the two measured bending laminations seen on the boundary of the convex core of $M$. See \cite{bonahon-otal} for the best known results in that direction.

Bending laminations are a generalization of dihedral angles. Thus, a question parallel to BLC is to ask whether the (exterior) dihedral angles $\alpha_1, \dots, \alpha_n$ of a convex hyperbolic polyhedron $P$ parametrize the deformation space of $P$. For compact $P$, the conditions on the $\alpha_i$ are complicated, non-linear even for the tetrahedron, and in general unknown (Andreev's theorem \cite{andreev} gives the deformation space, but only when all angles are acute). In contrast, for ideal and hyper-ideal $P$, one can prove \cite{bonahon-bao} there is only a finite set of linear constraints on the $\alpha_i$ (with second member), thus identifying the deformation space of $P$ with an affine polytope in some $\mathbb{R}^n$. This subject was inaugurated by Rivin \cite{rivin-1} and has been picked up by many authors since (see \cite{luo} and the references therein).

Let us focus on ideal $P$. Rivin's constraints stipulate that 
\begin{enumerate}[(I)]
\item the sum of the $\alpha_i$ associated to the edges around each vertex must be $2\pi$; 
\item the sum of the $\alpha_i$ associated to the edges of $P$ crossed by any transverse (non-backtracking) path $\gamma \subset \partial P$ must be larger than $2\pi$. 
\end{enumerate}
Since $\gamma$ bounds a disk, (II) mimics the Gauss--Bonnet inequality on the angles of a hyperbolic polygon (homeomorphic to the disk).

It is thus tempting to give the following definition: a \emph{topological polyhedron} is a pair $(M,G)$ where $M$ is a compact $3$-manifold with boundary and $G=(V,E)$ a non--empty graph embedded in $\partial M$, with vertex set $V$ and edge set $E$. An \emph{ideal realization} of $(M,G)$ (or in short: of $M$) is a complete hyperbolic manifold $M'$ with ideal polyhedral boundary (i.e. modelled on the intersection of two half--spaces of $\mathbb{H}^3$ and having finite volume), endowed with a homeomorphism to $M\smallsetminus V$ that sends the non--smooth points of $\partial M'$ precisely to the interior points of the edges of $G$.

Clearly, there may be topological obstructions to the existence of ideal realizations of $(M,G)$: for example, when $M$ is a ball, it is a theorem of Steiner that $G$ should be \emph{$3$--connected} for there to exist a convex polyhedron with $G$ as its $1$--skeleton. In general, $M$ must at least be assumed irreducible and atoroidal. Nevertheless, we can ask the natural question: 

\begin{question}
If the topological polyhedron $(M,G)$ has an ideal realization, is the space of such realizations $M'$ parametrized by the dihedral angles of $M'$, subject to Rivin's (linear) conditions?
\end{question}

Here, Rivin's conditions are the same as for the polyhedron case ($M\simeq\mathbb{B}^3$), except that only \emph{disk-bounding} curves $\gamma$ must be considered in (II).

In unpublished work \cite{schlenker}, Schlenker was able to answer this question in the affirmative when $\partial M$ is incompressible in $M$, via a sophisticated argument relying on the theorem of domain invariance. It seems the difficulty for the general case is to prove a certain \emph{connectivity} property of the union of all candidate deformations spaces $E_G$ when the graph $G$ varies. 

The aim of the present paper is to treat the case where $M$ is a solid torus, i.e. the simplest compressible case. Moreover, by slightly modifying (II), we will allow the core geodesic of $M$ to become a singular curve of fixed cone angle $K\in \mathbb{R}^+$ (degenerating to a rank-two cusp for $K=0$). The metric on $M$ will be Riemannian at the core precisely when $K=2\pi$. Our method is a generalization of \cite{gueritaud-polyhedra}, where we treated the case $M=\mathbb{B}^3$. Namely, $M$ is decomposed into ideal tetrahedra $\Delta_1, \dots, \Delta_m$ for which one finds positive dihedral angles $(\theta_1,\dots,\theta_{3m})$ satisfying a natural system of constraints $\Sigma$ (depending on the dihedral angles $\alpha_i$ of $M$). On the affine polytope $\Theta$ of such positive angle assignments $(\theta_1,\dots,\theta_{3m})$, the \emph{volume functional} $\mathcal{V}:\Theta\rightarrow \mathbb{R}^+$ turns out to have a unique critical point, which yields the hyperbolic metric on $M$ (i.e. the realization $M'$).

This approach was suggested by Rivin in \cite{rivin-1} and carried out independently in \cite{rivin-optimization} and \cite{gueritaud-polyhedra} in the (nonsingular) ball case, where the tetrahedra $\Delta_j$ are obtained by coning off all faces of $M$, triangulated if necessary, to one vertex. Here, we will use a \emph{spun} triangulation: in the infinite cyclic cover $\widetilde{M}$ of the solid torus $M$, which has two points at infinity $p^+$ and $p^-$, we cone all faces of $\widetilde{M}$ to $p^+$, and project the $\pi_1(M)$-equivariant result back to $M$. In the geometric realization, the tips of the tetrahedra $\Delta_j$ will thus \emph{spin} along the core geodesic of $M$.

The main difficulty, in the ball case as well as here, is to prove that $\Theta$ is not empty --- in other words, that the linear system $\Sigma$ has a \emph{positive} solution precisely when Rivin's conditions (I) -- (II) are satisfied. Here is an outline of the strategy: in the affine space $\widetilde{\Theta}$ of all (not necessarily positive) solutions $(\theta_1,\dots, \theta_{3m})$ to $\Sigma$, consider a point $\theta$ that \emph{minimizes} $\frac{1}{2}\sum_j|\theta_j|-\theta_j$, the sum of the negative parts of the $\theta_j$. If $\theta$ has a negative coordinate $\theta_j$, minimality will imply the presence of another negative coordinate living nearby in $\partial M$, and hence another, etc. Eventually this must produce a loop in $\partial M$, but Rivin's condition (II) says something about such loops, and we will get a contradiction. The increased difficulty for the solid torus comes from the fact that not \emph{all} loops are subject to (II) (i.e. bound disks): this forces us to study more closely exactly which $\theta_j$'s can become negative simultaneously and how the deformations in $T\widetilde{\Theta}$ may bring them back into $\mathbb{R}^{>0}$.

The paper is organized as follows. In Section \ref{statement}, we give all definitions and state the main theorem. In Section \ref{linear} we write out a linear program $\Sigma$, describe the full space $\widetilde{\Theta}$ of real (not necessarily positive) solutions using linear algebra from \cite{neumann-zagier}, and explain why finding a positive solution will imply the main theorem --- using facts about the volume functional from \cite{rivin-1}. In Section \ref{stokes} we give a ``discrete Stokes formula'' for graphs in $\partial M$. This is used in Section \ref{nonnegative} to work out strong topological constraints on solutions $(\theta_j)_{1\leq j \leq 3m}$: in particular, the nonpositive coordinates of the minimizing solution $\theta$ will define a \emph{train track} on the torus $\partial M$. Our description of $\widetilde{\Theta}$ in terms of curves in $\partial M$, coupled with a normalization result (Lemma \ref{traintracks}) for such curves with respect to train tracks, then allows us to conclude.

A very special case of the main result (with the graph $G\subset \partial M$ having only one vertex and three edges, and $K=2\pi$) was used and proved in Section 1 of \cite{dehn}, using the same approach as here.

\section{The main result}\label{statement}
Let $M$ be a solid torus, and consider a simplicial decomposition $\mathcal{O}=(F,E,V)$ of the boundary $\partial M$ into $m$ triangular Faces, $\frac{3}{2}m$ Edges and $\frac{m}{2}$ Vertices, with $m\geq 2$ even (since the torus $\partial M$ has Euler characteristic zero, all simplicial decompositions are of that type). Let $K \geq 0$ be a real number.

An embedded closed curve $\gamma \subset \partial M$ is called a transverse path if $\gamma$ is smooth, transverse to the $1$-skeleton of $\partial M$, and for every open (triangular) face $f\in F$ of $\partial M$, each connected component of $f\cap \gamma$ connects distinct edges of $f$.

Given a map $\alpha:E\rightarrow [0,\pi)$ and a transverse path $\gamma$, we call $\alpha$-weight of $\gamma$ the number $W^{\alpha}(\gamma):=\sum_{i=1}^k\alpha(e_i)$, where $e_1,\dots, e_k \in E$ is the full cyclic list of edges of $\partial M$ crossed by $\gamma$ (with possible repeats). The map $\alpha$ is called an (exterior) dihedral angle map with singularity $K$ if for each transverse path $\gamma$, the following ``Rivin--like'' conditions are satisfied:

\begin{equation} \label{alpha}
\left \{
\begin{array}{lll}
\text{(i)}&W^\alpha(\gamma)>0;& \\
\text{(ii)}&W^\alpha(\gamma)=2\pi & \text{if $\gamma$ is the loop around a vertex of $\partial M$;} \\
\text{(iii)}&W^\alpha(\gamma)>2\pi & \text{if $\gamma$ bounds a boundary--parallel disk $D$ in $M$} \\ && \text{but is not the loop around a vertex;} \\
\text{(iv)}&W^\alpha(\gamma)>K & \text{if $\gamma$ bounds a compression disk of $M$.} \\
\end{array}
\right .
\end{equation}

Note that any disk-bounding curve $\gamma$ falls either into (ii), (iii) or (iv).
Note also that (i) is equivalent to the following: the edges $\alpha^{-1}(0,\pi)$, i.e. those with nonzero weight, separate the multiply-punctured torus $\partial M \smallsetminus V$ into \emph{contractible} polygons. Thus, we can simulate nontriangular faces by setting certain angles $\alpha(e)$ to $0$ while preserving condition (i).

\begin{theorem} \label{main}
There exists an ideal realization $M_\alpha$ of $M$ with exterior dihedral angles $\alpha(e_i)$, and a conical singularity of angle $K$ along the core, if and only if $\alpha$ satisfies (\ref{alpha}). Moreover, $M_\alpha$ is then unique up to isometry, and the bijection $\alpha \mapsto M_\alpha$ between dihedral angle maps and the space of ideal realizations of $M$ (or moduli space of $M$) is a homeomorphism. 
\end{theorem}

It is easy to see that the reciprocal map $M_\alpha \mapsto \alpha$ is well-defined and continuous (at any rate the proof of Fact \ref{completion} below will imply it). The chief aim of this paper is to build the direct map.

\section{The linear problem}\label{linear}
We now describe a natural dissection of the solid torus $M$ into tetrahedra, and a linear problem attached to this dissection. An ideal hyperbolic tetrahedron $\Delta$ is defined up to isometry by its triple of dihedral angles $x,y,z>0$ such that $x+y+z=\pi$ (any two opposite edges carry the same angle).

Consider the decomposition of the torus $\partial M$ into triangles $\tau_1, \dots, \tau_m$. \emph{Cone} this decomposition to a point $\infty$ (the cone on each triangle is a tetrahedron), then remove the $0$--skeleton of $\partial M$ as well as the point $\infty$: now each $\tau_i$ is the base of a (topological) ideal tetrahedron $\Delta_i$. 

Number the $3m$ \emph{corners} of the $m$ triangles $\tau_1, \dots, \tau_m$ with labels $1$ through $3m$, and write formal variables $\theta_1, \dots, \theta_{3m}$ in these corners (the order does not matter). If $v, v', v''$ are (the ideal vertices at) the three corners of $\tau_i$ bearing formal variables $\theta_j, \theta_{j'}, \theta_{j''}$ respectively, we see $\theta_j$ (resp. $\theta_{j'}$, resp. $\theta_{j''}$) as the interior dihedral angle of $\Delta_i$ at the edges $\infty v$ and $v'v''$ (resp. $\infty v'$ and $vv''$, resp. $\infty v''$ and $vv'$).

Let $\mu \in H_1(\partial M, \mathbb{Z})$ denote the class of the meridian of $M$, and consider a transverse (embedded) path $\gamma\subset\partial M$ along the slope $\mu$, crossing edges $e_1,\dots, e_k$ in that cyclic order, with possible repeats. Any two consecutive edges $e_s$ and $e_{s+1}$ bound a triangle $\tau_s$, and we say that $\gamma$ makes a Left (resp. Right) at $\tau_s$ if $e_s, e_{s+1}$ appear clockwise (resp. counterclockwise) consecutively in that order along the boundary of $\tau_s$. For each $1\leq s \leq k$, let $\varepsilon_s$ be equal to $1$ if $\gamma$ makes a left at $\tau_s$, and to $-1$ if $\gamma$ makes a right at $\tau_s$. Let moreover $X_s \in \{1,2,\dots,3m\}$ be the label at the corner of $\tau_s$ where the edges $e_s, e_{s+1}$ meet, and $\theta_{X_s}$ be the associated formal variable.

Denote by $\Sigma$ the following system in the real variables $\theta_1, \dots, \theta_{3m}$:

\begin{equation} \label{sigma}
\left \{ 
\begin{array}{lrcll}
\text{(i)}&
\theta_j+\theta_{j'}&=&\pi-\alpha(e) & 
\left \{ \begin{array}{l}
\text{if $j, j'$ are the labels} \\
\text{opposite an edge $e$ of $\partial M$;}
\end{array} \right . \\
&&&\\
\text{(ii)}&
\theta_j+\theta_{j'}+\theta_{j''} &=&\pi & 
\left \{ \begin{array}{l}
\text{if $j,j',j''$ are the labels at the} \\  
\text{corners of a triangle $\tau_i$ of $\partial M$;}
\end{array} \right . \\
&&&\\
\text{(iii)}&
\theta_{j_1}+\dots+\theta_{j_t}&=&2\pi & 
\left \{ \begin{array}{l}
\text{if $j_1, \dots , j_t$ are the} \\
\text{labels around a vertex;}
\end{array} \right . \\
&&&\\
\text{(iv)}&
\sum_{s=1}^k\varepsilon_s\theta_{X_s} &=&K & 
\left \{ \begin{array}{l}
\text{with $\varepsilon_s\in\{-1,1\}$ and} \\
\text{$X_s \in \{1,\dots,3m\}$ as above.}
\end{array} \right . 
\\
\end{array}
\right .
\end{equation}

Condition (\ref{sigma}--i) is the natural criterion to make the exterior dihedral angle of $M$, at the edge $e$, equal to $\alpha(e)$: indeed $\theta_j$ and $\theta_{j'}$ are the angles of the two adjacent tetrahedra at their common edge $e$. Conditions (\ref{sigma}--ii--iii) are clearly required for a hyperbolic metric on a union of ideal tetrahedra.

It is a simple exercise to check that (\ref{sigma}--iv) does not depend on the chosen transverse path $\gamma$, but only on its slope $\mu$, because (\ref{sigma}--ii--iii) allow one to drag the path $\gamma$ across a vertex of $\partial M$ without variation of the left member of (\ref{sigma}--iv). The left member of (\ref{sigma}--iv) is called the (angular) \emph{holonomy} of $\gamma$ and, by extension, of its homology class $\mu$. The class $-\mu$ has holonomy $-K$.

The choice of $K$, rather than $-K$, as the right member of (\ref{sigma}--iv), is arbitrary, and corresponds to the choice of letting the tips of the ideal tetrahedra spin one way or the other along the core of the solid torus $M$. Thus, an interesting consequence of Theorem \ref{main} is that (\ref{main}) admits positive solutions if and only if the modified system (with $K$ changed to $-K$ in line iv) does.

\subsection{The solution space of $\Sigma$}

We begin with a construction which, to each transverse path $\gamma$, associates a tangent vector to the solution space of $\Sigma$. Lemma \ref{NZ} will state that all solutions can be reached by following such vectors, and in Section \ref{maximization} we will choose a particular solution to realize the solid torus $M$ geometrically.

Let $\Phi:\mathbb{R}^{3m}\rightarrow \mathbb{R}^{3m}$ be the linear map which takes the $j$-th basis vector $e_j$ to $(e_{j''}-e_{j'})$, where the labels $j,j',j'' \in \{1,\dots,3m\}$ appear counterclockwise in that order at the corners of a triangle $\tau$ of the torus $\partial M$ (this property characterizes $j',j''$ if $j$ is given).

Let $\gamma$ be an (oriented) transverse path in $\partial M$. We define the vector $T_\gamma \in \mathbb{R}^{3m}$ as follows: whenever $j_1,j_2,j_3 \in \{1,\dots,3m\}$ are the labels at the three corners of a triangle $\tau\subset \partial M$ (in any order), denote by $N_{j_1j_2j_3}$ the number of times the oriented path $\gamma$ crosses the triangle $\tau$ with entry point on the edge $j_1j_2$ and exit point on the edge $j_2j_3$. Then, 
$$T_\gamma:=\sum_{j=1}^{3m}\left ( N_{j'jj''}-N_{j''jj'}\right ) e_j~.$$
Thus, the $j$-th coordinate of $T_\gamma$ is the number of times $\gamma$ traverses the angular sector of the corner $j$, counted algebraically.

As a consequence, we find
\begin{eqnarray*} \Phi (T_\gamma)
&=& \sum_{j=1}^{3m} \left ( N_{j'jj''}-N_{j''jj'}\right ) (e_{j''}-e_{j'}) \\
&=& \sum_{j=1}^{3m} \left ( N_{j''j'j}-N_{jj'j''}\right ) e_j
  - \sum_{j=1}^{3m} \left ( N_{jj''j'}-N_{j'j''j}\right ) e_j \\
&=& \sum_{j=1}^{3m} \left ( N_{j''j'j}+N_{j'j''j}-N_{jj''j'}-N_{jj''j'}\right )e_j~.
\end{eqnarray*}
\begin{observation}\label{recount}
Thus, the $j$-th coordinate of $\Phi(T_\gamma)$ is the number of times the path $\gamma$ crosses the edge $j'j''$, counted positively for crossings towards $j$ (or into the triangle $\tau$), negatively for crossings away from $j$ (out of $\tau$).
\end{observation}

\begin{proposition} \label{deform}
For any transverse path $\gamma$ in the torus $\partial M$, if $\theta= (\theta_1,\dots,\theta_{3m})$ is a solution of \emph{(\ref{sigma}--i--ii--iii)}, then so is $\theta+\Phi(T_\gamma)$.
\end{proposition}
\begin{proof}
(Condition (\ref{sigma}--iv) is intentionally absent from this statement).
If $j,j',e$ are as in (\ref{sigma}--i), then the $j$-th and $j'$-th coordinate of $\Phi(T_\gamma)$ are each the algebraic number of times the path $\gamma$ crosses the edge $e$, counted with opposite signs: so they sum to $0$, and (\ref{sigma}--i) continues to hold at $\theta+\Phi(T_\gamma)$. The identity (\ref{sigma}--ii) clearly continues to hold even if $\theta$ is perturbed by a vector of the form $\Phi(e_s)$ (with $e_s$ any basis vector of $\mathbb{R}^{3m}$), because $0+1+(-1)=0$. Finally, (\ref{sigma}--iii) is redundant with (\ref{sigma}--i--ii): for a vertex $v$ of degree $k$, add up the $k$ equations of type (\ref{sigma}--ii) corresponding to the $k$ faces adjacent to $v$, and subtract the $k$ equations of type (\ref{sigma}--i) corresponding to the $k$ edges adjacent to $v$: the result is the equation of type (\ref{sigma}--iii) corresponding to $v$, and its right member is $2\pi$ by (\ref{alpha}--ii).
\end{proof}

It remains to find out which deformations additionally preserve (\ref{sigma}--iv). The following is a paraphrase of Theorem 2.2, Proposition 2.3 and Proposition 2.5 of \cite{neumann-zagier}.
\begin{lemma} \label{NZ}
Let $\gamma_1,\dots,\gamma_{\frac{m}{2}}$ be the loops around the vertices of the torus $\partial M$, and let $\gamma,\gamma'$ be transverse paths whose classes form a basis of $H_1(\partial M,\mathbb{Z})$, with $[\gamma]=\mu$, the meridian. Let $T$ be the $3m$-by-$(\frac{m}{2}+2)$ matrix with columns $T_{\gamma_1},\dots,T_{\gamma_{\frac{m}{2}}},T_{\gamma}, T_{\gamma'}$. Then the only linear relationship between the columns of $\Phi T$ is $$\sum_{i=1}^{m/2}\Phi(T_{\gamma_i})=0~,$$ and one has 
\begin{equation}
\label{symplec} 
{}^tT\,\Phi\, T = \left ( \begin{array}{cccc} 
0 & \cdots  & 0 & 0 \\
\vdots & \ddots  & \vdots & \vdots \\
0 & \cdots & 0 & \pm 2 \\
0 & \cdots & \mp 2 & 0 \end{array}\right )~.
\end{equation}
\end{lemma}
The proof is the same as in \cite{neumann-zagier} (the fact that some faces of our ideal tetrahedra are not paired up does not change anything essential).

\begin{corollary} \label{gentang}
The system (\ref{sigma}) has real solutions, and the tangent space to the solution space $\widetilde{\Theta}$ is $\frac{m}{2}$-dimensional, generated by $\Phi T_{\gamma_1},\dots,\Phi T_{\gamma_{\frac{m}{2}}}, \Phi T_{\gamma}$.
\end{corollary}
\begin{proof}
First, (\ref{sigma}) involves $3m$ unknowns, and $(\frac{5m}{2}+1)$ relations (one for each of the $\frac{3m}{2}$ edges (\ref{sigma}--i), one for each of the $m$ triangles (\ref{sigma}--ii), plus (\ref{sigma}--iv); we can omit (\ref{sigma}--iii) because they are redundant with the rest).

Let us now check that there is only one relationship (up to scalar multiplication) between the conditions (\ref{sigma}--i) and (\ref{sigma}--ii). In a linear combination of equations (\ref{sigma}--i), the coefficient of the variable $\theta_j$ depends only on the edge opposite the corner labelled $j$. In a linear combination of equations (\ref{sigma}--ii), the coefficient of $\theta_j$ depends only on the triangle to which the corner $j$ belongs. If these two linear combinations are equal, it is thus easy to see that all $\theta_j$ have the same coefficient (because $\partial M$ is connected). Therefore, the only relationship is that the sum of \emph{all} equations (\ref{sigma}--i) is equal to the sum of \emph{all} equations (\ref{sigma}--ii). We only need to check that the right members agree. If $e \succ v$ denotes the property for an edge $e$ to be incident to a vertex $v$, then the right members of (\ref{sigma}--i) add up to $$\frac{3m}{2}\pi-\sum_{e\in E}\alpha(e)=\frac{3m}{2}\pi-\frac{1}{2}\sum_{v\in V}\sum_{e \succ v} \alpha(e)=\frac{3m}{2}\pi-\frac{1}{2}\left (\frac{m}{2}\cdot 2\pi\right )=m\pi$$
by (\ref{alpha}--ii) (in the double sum, count $\alpha(e)$ twice if $e$ is incident to $v$ at both ends), while the right members of (\ref{sigma}--ii) clearly also add up to $m\pi$.

Therefore the solution space $\Theta'$ of (\ref{sigma}--i--ii)=(\ref{sigma}--i--ii--iii) has dimension exactly $3m-\frac{5m}{2}+1=\frac{m}{2}+1$. By Lemma \ref{NZ} and Proposition \ref{deform}, the tangent space of $\Theta'$ is generated by $\Phi T_{\gamma_1},\dots,\Phi T_{\gamma_{\frac{m}{2}}},\Phi T_{\gamma},\Phi T_{\gamma'}$, the column vectors of $\Phi T$ (the first $\frac{m}{2}$ of which sum to $0$).

Finally, note that (\ref{sigma}--iv) can be written: ${}^tT_{\gamma}\theta=K$. This is a codimension--1 condition in $\Theta'$, defining $\widetilde{\Theta}$: more precisely, by looking at the next-to-last line of (\ref{symplec}), all column vectors of $\Phi T$ belong to $\ker {}^tT_{\gamma}$, except $\Phi T_{\gamma'}$. Therefore ${}^tT_{\gamma}\theta$ can be set to any value by just perturbing $\theta$ by a multiple of $\Phi T_{\gamma'}$. In conclusion, the tangent space to the solution space $\widetilde{\Theta}$ of (\ref{sigma}) is generated by $\Phi T_{\gamma_1},\dots,\Phi T_{\gamma_{\frac{m}{2}}},\Phi T_{\gamma}$, where $[\gamma]$ is the meridian of the solid torus $M$.
\end{proof}

\subsection{Volume maximization} \label{maximization}

Let us now assume that the positive solution space $\Theta := \widetilde{\Theta}\cap\mathbb{R}_{>0}^{3m}$ is non--empty. Then, Theorem \ref{main} holds. The argument is as follows (we essentially follow Rivin \cite{rivin-1}: the context is only slightly different here because of the presence of (\ref{sigma}--iv)).

We call a point of $\Theta$ an \emph{angle structure}. Define the volume functional 
\begin{equation} \label{volfunc}
\begin{array}{rccl} &\Theta & \longrightarrow & \mathbb{R}_+^* \\ \mathcal{V}~: & (\theta_1,\dots,\theta_{3m})& \longmapsto & \displaystyle{-\sum_{s=1}^{3m} \int_{0}^{\theta_s} \log |2\sin t| \, dt~.} \end{array}\end{equation} 
(This functional associates to an angle structure the sum of the volumes of the corresponding ideal tetrahedra, by the famous Lobatchevski formula \cite{lobatchevski}.) 

\begin{fact} If $\Theta\neq\emptyset$, then $\Theta$ contains a unique critical point of $\mathcal{V}$. \end{fact}
 \begin{proof} First, $\mathcal{V}$ is strictly concave on $\Theta$ and can be continuously extended to the (compact) closure $\overline{\Theta}$ of $\Theta$ in $\mathbb{R}^{3m}_{\geq 0}$ (see for example \cite{gueritaud-futer}, Prop. 6.6). Therefore, we only need to check that any maximum $\theta^{\text{max}}$ of $\mathcal{V}$ in $\overline{\Theta}$ belongs to $\Theta$. A useful fact is that at $\theta^{\text{max}}$, any triangle $\tau$ of $\partial M$ that has an angle $\theta^{\text{max}}_i=0$ has angles $(\pi,0,0)$ up to order (otherwise $\mathcal{V}$ has derivative $+\infty$ in some direction at $\theta^{\text{max}}$, see for example \cite{gueritaud-futer}, Prop. 6.7). Such triangles are called \emph{flat}. By (\ref{sigma}--i), since $\alpha\geq 0$, the edge $e$ of $\tau$ opposite the angle $\pi$ must then satisfy $\alpha(e)=0$, and the angle opposite $e$ in the neighbor $\tau'$ of $\tau$ is also $0$. But then, $\tau'$ is also flat and we can repeat the argument, eventually finding a cyclic sequence $\tau,\tau',\tau'',\dots,\tau^{(s)}=\tau$ of flat triangles, such that in particular $\alpha(e)=0$ for any edge $e$ separating two consecutive triangles. This contradicts (\ref{alpha}--i). It follows that no angle at $\theta^{\text{max}}$ can vanish, so $\theta^{\text{max}}\in \Theta$. \end{proof}

\begin{fact} If $K>0$, then the angle structure $\theta=\theta^{\text{max}}$ determines an incomplete hyperbolic metric on the union of tetrahedra $\Delta_1\cup \dots \cup \Delta_m$. The completion is a complete but singular hyperbolic metric on the solid torus $M$ (minus the $\frac{m}{2}$ vertices of $\partial M$). The incompleteness locus becomes the singular locus $L$ after completion, and $L$ is a closed geodesic, isotopic to the core of $M$, surrounded by a cone angle of measure $K$. If $K=0$, then $\theta=\theta^{\text{max}}$ determines a complete hyperbolic metric on $\Delta_1\cup \dots \cup \Delta_m$, homeomorphic to the solid torus $M$ minus its core curve (in particular the core curve is replaced by a rank-2 cusp).
\label{completion} \end{fact}

\begin{proof}
Suppose first that $K>0$.
Recall that around each vertex $v$ of the triangulated torus $\partial M$ runs a transverse path $\gamma_i$. By Corollary \ref{gentang}, $\Phi T_{\gamma_i}$ belongs to the tangent space of $\Theta$, so $\mathcal{V}$ must be critical in the direction $\Phi T_{\gamma_i}$. Plugging this into (\ref{volfunc}), if $v$ has degree $k$, we get
\begin{equation} \label{sineratio}
{\frac{d}{d\varepsilon}}_{|\varepsilon=0}\mathcal{V}(\theta+\varepsilon\cdot\Phi(T_{\gamma_i}))=\sum_{s=1}^k \log \sin \theta_{j_{2s}}-\log\sin\theta_{j_{2s-1}}=0
\end{equation}
where for convenience we assumed that the labels at the corners of the $k$ triangles adjacent to $v$ are numbered $j_1$ through $j_{2k}$, counterclockwise (disregarding the $k$ labels at the vertex $v$ itself). By the sine formula of Euclidean trigonometry, $\prod_{j=1}^k\frac{\sin \theta_{j_{2s}}}{\sin \theta_{j_{2s-1}}}=1$ means that the $k$ ideal tetrahedra obtained by coning these $k$ triangles to $\infty$ glue properly, yielding a complete hyperbolic metric near $v\infty$ (see e.g. \cite{chan-hodgson} for a more detailed argument).

We therefore get a \emph{developing map} $f$ from the universal cover $X$ of the torus $\partial M$, to $\mathbb{R}^2\simeq \mathbb{C}$. This map induces a representation $\rho$ from $\pi_1(\partial M)\simeq \mathbb{Z}^2$ to the group $\mathbb{C}^*\ltimes \mathbb{C}$ of similarities of $\mathbb{C}$. Since $\pi_1(M)$ is commutative, the image of $\rho$ is contained either in $\{1\}\ltimes \mathbb{C}$ (a lattice of translations), or in $\mathbb{C}^*\ltimes \{0\}$ (a group of multiplications, fixing $0$ without loss of generality). In the first case, the left member of (\ref{sigma}--iv) would have to be $0$, but we asumed $K>0$. Therefore we are in the second case. Writing out that $\mathcal{V}$ is critical in the direction $\Phi T_{\gamma}$ (for $\gamma$ representing the meridian), we find \begin{equation}\label{sineratio2}{\frac{d}{d\varepsilon}}_{|\varepsilon=0}\mathcal{V}(\theta+\varepsilon\cdot\Phi(T_{\gamma}))=\log\prod_{s=1}^k\frac{\sin \theta_{j_{3s+1}}}{\sin \theta_{j_{3s+2}}}=0~,\end{equation}
where we assumed the meridinal curve $\gamma$ enters the $s$-th triangle $\tau_s$ through the edge connecting the corners $j_{3s}$ and $j_{3s+1}$, and leaves $\tau_s$ through the edge connecting the corners $j_{3s}$ and $j_{3s+2}$. Using the sine formula as above, we see that the similarity $\rho(\gamma)$ has a scaling factor of $1$. By (\ref{sigma}--iv), $\rho$ also has a rotational component of angle $K$, i.e. $\rho$ is a rotation of angle $K$. The developing map $f$ therefore induces a map $f_K$ taking $\gamma$ to a closed (Jordan) curve of $\mathbb{R}^2_K$, defined as the Euclidean plane with a $K$-singularity at the origin. Similarly, $\rho$ induces a map $\rho_K$ from $\pi_1(\partial M)\simeq \mathbb{Z}^2$ to the similarities of $\mathbb{R}^2_K$, and the meridian is in the kernel of $\rho_K$. The complete, singular metric on the solid torus $M$ is now obtained by gluing together all the ideal tetrahedra living above the triangles in the image of $f_K$ (in the singular upper half--space model above $\mathbb{R}^2_K$), adding the singular line $0\infty$, and quotienting out by the loxodromy induced by the similarity $\rho_K(\gamma')$.

Importantly, the argument reverses: \emph{if} we are given a hyperbolic metric on the solid torus $M$ with dihedral angles $\alpha$ and $K$-singular core, then the induced spun triangulation, which is well-defined, comes naturally with an angle structure $\theta$ which must satisfy (\ref{sineratio}) and (\ref{sineratio2}), i.e. $\theta$ must be critical for $\mathcal{V}$. By concavity of $\mathcal{V}$, the hyperbolic metric on $M$ is thus unique and depends continuously on the angles $\alpha$ and on $K$: when a strictly concave function changes continuously, its maximum moves around continuously. Theorem \ref{main} is proved (assuming $\Theta \neq \emptyset$) for $K>0$.

Finally, if $K=0$, then $\rho(\gamma)$ is a translation, so by commutativity the whole image of $\rho$ is a lattice of translations of $\mathbb{C}$. The union of all ideal tetrahedra, built above the triangles of the image of the developing map $f$, is already complete. Its quotient by $\text{Im}(\rho)$ has a rank-two cusp at infinity. Uniqueness of the metric is proved as before.
\end{proof}

We finish this section with two remarks: first, (\ref{alpha}) provides a parametrization of the deformation space of $M$ as both its dihedral angles $\alpha(e_i)$ \emph{and} its conical singularity angle $K$ vary, and the volume $V=\mathcal{V}(\theta^{\text{max}})$ is a concave function of all these variables. Second, a consequence of the uniqueness result above is that for every $K\geq 0$, if (\ref{sigma}) has positive solutions, so does (\ref{sigma}) with $K$ changed to $-K$ in line (iv). More particularly, the function $V$ is even in the variable $K$.
If we change (\ref{sigma}--iv) to $\sum_{s=1}^k\varepsilon_s\theta_{X_s} \geq K$ (or $\leq -K$), the volume maximizer $\theta^{\text{max}}$ under this new set of constraints still provides the unique geometric realization of $M$.

In the remainder of the paper, we suppose (\ref{alpha}) and prove that $\Theta \neq \emptyset$, i.e. we find a positive solution of (\ref{sigma}).

\section{The discrete Stokes formula} \label{stokes}
Recall the tessellation $\mathcal{O}=(V,E,F)$ of the torus $\partial M$ defined by the vertices, edges and complementary triangles of the graph $G$. Dual to this tessellation is a (topological) tessellation $\mathcal{O}'=(V',E',F')$ of $\partial M$ such that there are canonical bijections $V'\simeq F$; $E'\simeq E$; $F' \simeq V$ reversing all cell inclusions. All these bijections, as well as their inverses, will be denoted by a star in the superscript. We view $\mathcal{O}$ and $\mathcal{O}'$ as simultaneously drawn on the torus $\partial M$. There is a natural bijection $$\iota:\overrightarrow{E}'\rightarrow \{1,\dots,3m\}$$ from the set of \emph{oriented} edges of $\mathcal{O}'$ to the set $\{1,\dots,3m\}$ of corners of the triangles of $\mathcal{O}$: the corner of $\tau$ opposite $\tau'$ corresponds, by $\iota^{-1}$, to the oriented edge from the vertex ${\tau'}^*$ to the vertex $\tau^*$ of $\mathcal{O}'$. For each $j\in \{1.\dots,3m\}$, one may thus see $\iota^{-1}(j)$ as an arrow pointing towards the corner $j$ across the edge opposite $j$.

Consider an arbitrary subcomplex $\Gamma$ of $\mathcal{O}'$. Define the \emph{discrete boundary} $\Delta \Gamma$ of $\Gamma$ to be the union of all oriented edges of $\mathcal{O}'$ pointing to a vertex of $\Gamma$, but not included (after forgetting the orientation) in $\Gamma$.

\begin{lemma} \label{lemstokes}
Suppose $\theta=(\theta_1,\dots,\theta_{3m})$ is a solution of (\ref{sigma}--i--ii--iii). Then, for every subcomplex $\Gamma$ of $\mathcal{O}'$, one has 
$$\sum_{\varepsilon \in \Delta \Gamma}\theta_{\iota(\varepsilon)}= \pi\cdot \chi(\Gamma)+ \frac{1}{2}\sum_{e \in \partial \Gamma} \alpha(e^*)~.
$$
\end{lemma}
Here, $\chi$ is the Euler characteristic, and the second sum is taken over edges $e \in E'$ which belong to the boundary (in the usual, topological sense) of $\Gamma$, counted twice if none of the two faces adjacent to $e$ belongs to $\Gamma$. (Equivalently, the second sum could be taken over pairs $(e, f) \in E' \times F'$ such that the edge $e$ belongs to $\Gamma$, the face $f$ does not, and $f$ is adjacent to $e$ --- counted twice if $f$ is adjacent to $e$ on both sides.)

\begin{proof}
We prove Lemma \ref{lemstokes} by induction on $\Gamma$. 

If $\Gamma$ is its own $0$--skeleton, consisting of $k$ vertices, then the first sum contains $3k$ terms summing to $k\pi$, by (\ref{sigma}--ii). In the right member, the sum is empty and $\chi(\Gamma)=k$, so the identity holds.

If $\Gamma$ is its own $1$--skeleton, consider the situation where an edge $e \in E'$ is added to $\Gamma$, between two pre--existing vertices of $\Gamma$. The left member then drops by $\pi-\alpha(e^*)$, by (\ref{sigma}--i). In the right member, $\chi(\Gamma)$ drops by one, and the sum contains two new terms, both equal to $\alpha(e^*)$: the relation is preserved.

Finally, for general $\Gamma$, consider the situation where a face $f$, bounded by $k$ edges already in $\Gamma$, is added to $\Gamma$. The left member is then unchanged; in the right member, $\chi(\Gamma)$ increases by one, while the sum loses $k$ terms adding up to $2\pi$ by (\ref{alpha}--ii): again, the identity is preserved.
\end{proof}
One can see Lemma \ref{lemstokes} as a sort of discrete Stokes formula, relating the Euler characteristic of $\Gamma$ to information on its boundary.

\section{Finding a positive solution}\label{nonnegative}

On the space $\widetilde{\Theta}$ of all real solutions to (\ref{sigma}), consider a point $\theta$ that minimizes $\sigma(\theta):=\frac{1}{2}\sum_{j=1}^{3m}|\theta_j|-\theta_j$, the sum of the negative parts of the $\theta_j$ (such a point exists because $\sigma$ is clearly a proper function from $\widetilde{\Theta}$ to $\mathbb{R}_+$). Among all such points, we can moreover assume that $\theta$ minimizes the number $\nu$ of indices $j$ such that $\theta_j\leq 0$.
In this section, we will show that $\nu=0$, which proves Theorem \ref{main}. Suppose $\nu>0$ and aim for a contradiction.

\begin{definition} \label{arrows}
Let $Y'\subset \overrightarrow{E}'$ be the set of all oriented edges (or ``arrows'') $\varepsilon$ of the tessellation $\mathcal{O}'$ such that $\theta_{\iota(\varepsilon)}\leq 0$. Let also $Y\subset \overrightarrow{E}$ be the set of all oriented edges of $\mathcal{O}$ obtained by rotating the oriented edges in $Y'$ by a quarter turn, counterclockwise.
\end{definition}

\begin{claim} \label{extend}
For any vertex $v$ of the tessellation $\mathcal{O}$, the set $Y$ contains an oriented edge pointing to $v$ if and only if $Y$ contains an oriented edge pointing away from $v$.
\end{claim}
\begin{proof}
Suppose there exists a vertex $v$ with at least one incoming edge and no outgoing edge of $Y$. This means that the arrows in $Y'$, across edges adjacent to $v$, all turn counterclockwise as seen from $v$. In particular, no edge of $\mathcal{O}$ crossed by an arrow of $Y'$ connects $v$ to itself.

Any one of the $k$ triangles adjacent to $v$ has three corners, which we call the $v$--corner, the left corner, and the right corner (arising counterclockwise in this order): we may assume that the left corners are labelled $j_1,j_3,\dots,j_{2k-1}$ and the right corners are labelled $j_2,j_4,\dots, j_{2k}$, with $j_1,\dots,j_{2k}$ arising counterclockwise in that order around $v$. (Some right corners can be left corners at the same time, if there is an edge from $v$ to itself, but such an edge carries no arrow.) By assumption, all arrows in $Y'$ across edges adjacent to $v$ are then of the form $\iota^{-1}(j_{2s})$, and none are of the form $\iota^{-1}(j_{2s-1})$.

Recall the transverse path $\gamma_i$ around $v$: up to sign, the deformation $\lambda\cdot \Phi(T_{\gamma_i})$ consists in adding $\lambda$ to all the $\theta_{j_{2s}}$ and subtracting $\lambda$ from all the $\theta_{j_{2s-1}}$. But such an operation decreases either the sum $\sigma(\theta)$ of the negative parts, or (failing that) the number of nonpositive coordinates of $\theta$: this violates our assumption on $\theta$.

If there is at least one outgoing edge of $Y$ at $v$ and no incoming edge, the proof is the same up to exchanging even and odd indices.
\end{proof}

As a consequence, any oriented edge in $Y$ can be extended beyond both ends to an oriented path of arbitrary length carried by $Y \subset \overrightarrow{E}$. However, the following statement places strong restrictions on the set of oriented edges $Y$.

\begin{proposition} \label{nodisks}
Let $\beta$ be an oriented simple closed curve in the $1$--skeleton of $\mathcal{O}$, and suppose $\beta$ (with its orientation) is carried by $Y$. Then $\beta$ does not bound a disk in the solid torus $M$.
\end{proposition}
\begin{proof}
Let $Y'_{\beta}\subset Y'$ denote the set of arrows of $Y'$ lying transversely across edges of $\beta$. We first deal with boundary-parallel disks: if $\beta$ bounds a disk in $\partial M$, let $\Gamma$ be the maximal subcomplex of the tessellation $\mathcal{O}'$ contained in the component of $\partial M \smallsetminus \beta$ bearing the heads of the arrows of $Y'_{\beta}$. Then $\Gamma$ has the homotopy type of either a disk ($\chi=1$), or a torus minus a disk ($\chi=-1$). Moreover, we have $\Delta\Gamma=Y'_{\beta}$ by construction, so Lemma \ref{lemstokes} yields 
$$0\geq \sum_{\varepsilon \in Y'_{\beta}}\theta_{\iota(\varepsilon)}= \pi\cdot \chi(\Gamma)+\frac{1}{2}\sum_{e \in \partial \Gamma} \alpha(e^*)>0$$
(absurd!) where the right inequality comes from (\ref{alpha}--ii--iii) and the fact that $\partial \Gamma$ cannot be just the loop around a vertex if $\chi(\Gamma)=-1$: indeed, $\partial M \smallsetminus \Gamma$ should at least contain the edges of $\beta$.

Suppose now that $\beta$ bounds a compression disk of the solid torus $M$. If $\beta$ has length $b$ then $\beta$, being simple, has $b$ distinct edges and $b$ distinct vertices. We may assume that two arrows of $Y'_{\beta}$ lying across two consecutive edges of $\beta$ never point into the same triangle of $\mathcal{O}$: indeed, if two such arrows point towards the corners $j$ and $j'$ of a certain triangle, then $\theta_j\leq 0,$ and $\theta_{j'}\leq 0$, so the third corner $j''$ must satisfy $\theta_{j''}\geq \pi$ by (\ref{sigma}--ii); therefore the corner $j'''$, opposite $j''$ in the neighboring triangle, satisfies $\theta_{j'''}\leq 0$ by (\ref{sigma}--i). In other words, we can shorten $\beta$ by one, bypassing the two arrows to $j,j'$ and replacing them by the arrow to $j'''$.

Consider now the triangles immediately adjacent to $\beta$, on the side of the heads of the transverse arrows. These triangles fall into two families: $b$ of them intersect $\beta$ along a single edge, while some other number $a$ of them intersect $\beta$ only at a vertex. There is a transverse path $\gamma_0$ of length $c:=a+b$ making $b$ rights and $a$ lefts, and crossing precisely these triangles. The $3c$ corners of these $c$ triangles fall into three families:
\begin{itemize}
\item The $b$ corners opposite $\beta$ (i.e. the corners near which $\gamma_0$ makes Rights) form the first family, $B=\iota(Y'_{\beta})$.
\item The $a$ corners near which $\gamma_0$ makes Lefts form the second family, $A$.
\item The $2c$ remaining corners, the ones opposite the $c$ edges crossed by $\gamma_0$, form the third family, $C$.
\end{itemize}
We have the following properties:
\begin{equation} \label{equator}
\begin{array}{ccll}
\displaystyle{\sum_{j\in B} \theta_j} &\leq& 0&\text{by definition of $Y'_{\beta}$;} \\
\displaystyle{\sum_{j\in C} \theta_j} &<& c\pi-K &\text{by (\ref{sigma}--i) and Rivin's condition (\ref{alpha}--iv);} \\
\displaystyle{\sum_{j\in A\sqcup B\sqcup C} \theta_j} &=& c\pi&\text{by (\ref{sigma}--ii).}
\end{array}
\end{equation}
This immediately implies $\sum_{j\in A}\theta_j >K$. Hence, the holonomy of $\gamma_0$ is, up to sign, $\sum_{j\in A} \theta_j - \sum_{j\in B} \theta_j >K$: but if $\gamma_0$ bounds a compression disk, its holonomy must be $\pm K$ by (\ref{sigma}--iv). Therefore $\gamma_0$, and similarly $\beta$, does not bound a compression disk. (In fact, the same argument also works for the boundary-parallel disks that we treated first, with $2\pi$ instead of $K$.)
\end{proof}

We now convert the set $Y$ of oriented edges in $M$, into a train track, as follows.
At every vertex $v$ of the unoriented support $|Y|$ of $Y$, there are at least one incoming and one outgoing edge of $Y$. If there are only two edges total, we just consider $v$ as a generic point of the train track. If there are three or more edges, then there exists an integer $k_v \geq 1$ such that the (half--) edges of $Y$ adjacent to $v$ fall into $k_v$ maximal families of consecutive incoming edges, and $k_v$ maximal families of consecutive outgoing edges (\emph{consecutive} is with reference to the cyclic order around $v$, restricted to $|Y|$). We then replace $v$ with a spiky $(2k_v)$-gon, and attach each of the aforementioned families of $s$ consecutive edges to one of the vertices of the $(2k_v)$-gon, turning that vertex into a $2$-to-$s$ (or $s$-to-$2$) switch of the train track. Note that the train track, noted $\mathcal{T}$, inherits an orientation from $Y$.

\begin{proposition} \label{bigons}
If $\mathcal{T}\neq \emptyset$, then all complementary regions of $\mathcal{T}$ are either bigons or annuli with smooth boundary. In particular, all the integers $k_v$ are equal to $1$.
\end{proposition}

\begin{proof}
Each complementary region $R_i$ of $\mathcal{T}$ has a natural Euler characteristic $\chi_i$, equal to its topological Euler characteristic, minus half its number of vertices (spikes): for example, an $s$-gon gets $\chi_i=1-s/2$. These $\chi_i$ must add up to $\chi(\partial M)=0$. However, the number of spikes of any region is always even, because $\mathcal{T}$ is oriented. Therefore the only possible case where $\chi_i>0$ is when $R_i$ is a smooth disk ($0$-gon), but that is ruled out by Proposition \ref{nodisks}. It follows that $\chi_i=0$ for all regions $R_i$: the regions are either smooth annuli or bigons.
\end{proof}

Since all the integers $k_v$ are $1$, we alter the train track $\mathcal{T}$ again by collapsing all the bigons coming from vertices $v$ back to points. The oriented train track $\mathcal{T}$ is now again isotopic to the union of the original set of oriented edges $Y$, only smoothened near the vertices.

\begin{proposition} \label{noreeb}
The oriented train track $\mathcal{T}$ does not carry two parallel curves with opposite orientations (i.e. does not have any ``Reeb components'').
\end{proposition}

\begin{proof}
If there is a Reeb component bounded by two parallel curves $\gamma_0$ and $\gamma_1$ of $\mathcal{T}$, then $\gamma_0 \cup \gamma_1$ separates the torus $\partial M$ into two annuli, one of which (say $A$) contains all the heads of the arrows of $Y'$ across $\gamma_0$ and $\gamma_1$. Let $\Gamma$ be the maximal subcomplex of the dual tessellation $\mathcal{O}'$ contained in $A$: since $A$ retracts to $\Gamma$, we have $\chi(\Gamma)=0$. Recall the discrete Stokes formula: since $\Delta\Gamma$ is precisely the set of all arrows across $\gamma_0$ and $\gamma_1$, Lemma \ref{lemstokes} gives $\sum_{e\in \partial \Gamma}\alpha(e^*)\leq 0$. This is clearly inconsistent with (\ref{alpha}--i): the sum should be positive. As a consequence, there can be no Reeb components in $\mathcal{T}$. 
\end{proof}

Propositions \ref{nodisks}, \ref{bigons}, \ref{noreeb} put very strong constraints on $\mathcal{T}$. In fact,

\begin{lemma} \label{traintracks}
Consider an oriented train track $\mathcal{T}$ on the smooth torus such that all complementary regions are either bigons or smooth annuli, and $\mathcal{T}$ has no Reeb components. Then, up to isotopy, any oriented simple closed curve $c$ in the torus is either carried by $\mathcal{T}$ (possibly after reversing its orientation), or intersects $\mathcal{T}$ with the same sign everywhere. 
\end{lemma}
(The statement about the signs of the intersections just means, in this context, that $c$ never enters and leaves a bigon, or annulus, through the same edge: we will say that $c$ is \emph{transverse}.)

\begin{proof}
The following proof will actually give more: namely, in the first homology $\mathbb{Z}^2$ of the torus, the directions that are carried form a centrally symmetric cone (possibly reduced to a line) and all other directions are transverse. Moreover, the two extremal directions of the cone are obtained by following $\mathcal{T}$ and making rights (resp. lefts) at every switch of the train track.

Let us first refine the definitions: a smooth oriented curve $c$ in the torus is
\begin{itemize}
\item plus-carried (resp. minus-carried) if $c$ (resp. $c$ with reversed orientation) follows the oriented train track $\mathcal{T}$;
\item right-transverse (resp. left-transverse) if $c$ is transverse to $\mathcal{T}$ and, at each intersection, travels to the right (resp. left) if the direction of $\mathcal{T}$ is ``up''. Every transverse curve is either left-transverse or right-transverse.
\end{itemize}

We will say that a first homology class is carried, transverse, etc... whenever it is represented by a curve with that property. Given two plus-carried curves $c_1, c_2$, the first homology class $[c_1]+[c_2]$ is also plus-carried, as well as all its integer submultiples, by an easy surgery argument on $c_1$ and $c_2$. Similarly, ``minus-carried'' is also a stable property under homological sum. ``Right-transverse'' is also stable under homological sum (still by surgery, because all right-transverse curves enter any given bigon or annulus through the same edge) as well as ``left-transverse''.

Starting from any branch of $\mathcal{T}$, follow it while always turning right at the switches: after finite time this process converges to a cycle of $\mathcal{T}$ that we call $C_r$. Similarly, by always turning left at the switches, one converges to another cycle $C_l$. (It would be easy to check that while the curves $C_r, C_l$ may depend on the initial branch chosen, their slopes do not.) Clearly, $[C_r]$ and $[C_l]$ are plus-carried. Moreover, $[C_r]$ is also right-transverse: to see this, first notice that all branches of $\mathcal{T}$ incident to $C_r$ on the right are coming into (not out of) $C_r$; then push $C_r$ slightly to the right. Similarly, $[C_l]$ is left-transverse.

\smallskip

First, suppose $C_r$ and $C_l$ are not parallel, but intersect once or more essentially. Since $[C_r]$ and $[C_l]$ are both plus-carried, all (integer submultiples of) nonzero classes of the form $m[C_r]+n[C_l]$, with $m,n\geq 0$, are plus-carried. Their negatives are minus-carried. Since $[C_r]$ and $-[C_l]$ are both right-transverse, all (integer submultiples of) nonzero classes of the form $m[C_r]+n[C_l]$, with $m\geq 0\geq n$, are right-transverse. Their negatives are left-transverse. This covers all classes.

\smallskip

Next, suppose $C_r$ and $C_l$ have the same slope $\lambda$ (and therefore the same orientation, by Proposition \ref{noreeb}). The plus-carried slope $\lambda=[C_r]=[C_l]$ is both right-transverse and left-transverse. So is the minus-carried slope $-\lambda$. We shall prove that all other slopes are transverse.

First consider the case when the complement of the train track $\mathcal{T}$ contains an annulus $\mathcal{A}$ (also of slope $\lambda$).  By additivity, it is clearly enough to show that there exists a transverse arc connecting the two boundary curves of the complement $\mathcal{B}$ of $\mathcal{A}$ in the torus (this arc can then be closed up across $\mathcal{A}$, staying transverse). We can assume that inside $\mathcal{B}$, all complementary components of $\mathcal{T}$ are bigons: should there be any annuli, just repeat the argument. We will prove by induction on the number of branches of the train track on $\mathcal{B}$ that there is a transverse path across $\mathcal{B}$ starting from \emph{any} regular point $p$ of $\partial \mathcal{B}$. 

Note that $\mathcal{B}$ is topologically an annulus, possibly collapsed to its core curve in some places (so the initial case of the induction, where $\mathcal{B}$ is reduced to a circle, is trivial). The orientations on the two boundary curves of $\mathcal{B}$ (possibly merged with the core in some places) are the same. We can furthermore assume that $\mathcal{B}$ carries no other closed curve than these two boundaries: if $\mathcal{B}$ carries another closed curve $c$, split $\mathcal{B}$ along $c$ into two (possibly even more collapsed) annuli and use induction. We refer to the two boundary curves of $\mathcal{B}$ as the \emph{right} and \emph{left} one, with respect to the orientation of the train track $\mathcal{T}$.

Consider a boundary curve $c$ of $\mathcal{B}$. There is at least one arc of $\mathcal{T}$ leaving $c$ into $\mathcal{B}$, or reaching $c$ out of $\mathcal{B}$ (otherwise $\mathcal{B}$ contains an annulus next to $c$). Moreover, these two types of arcs cannot coexist, otherwise $\mathcal{B}$ carries a nonboundary curve (or $\mathcal{B}$ is just one bigon with an arc connecting the vertices, in which case we are done). Up to switching the orientation, we can therefore assume that all arcs of $\mathcal{B}$ incident on $c$, leave $c$ into $\mathcal{B}$ (to the right, if $c$ is the left boundary curve) for the orientation of the train track.

Starting from any regular point $p$ of the boundary curve $c$, construct an arc $\alpha$ by following the train track and making rights at every switch, until $\alpha$ becomes trapped in the cycle $C_r$. The closed curve $C_r$ can only be the  other boundary curve, $c'$ (because $\alpha$ cannot return to $c$). Cut off $\alpha$ the moment it reaches $c'$. On the right side of $\alpha$, all incident branches come into (not out of) $\alpha$, by construction. Therefore we can push $\alpha$ slightly to its right, with initial point $p$ fixed, to obtain a right-transverse curve from $p$ to $c'$. This finishes the proof when the complement of the oriented train track $\mathcal{T}$ contains an annulus.

\smallskip

Still assuming $[C_r]=[C_l]=\lambda$, suppose now that all complementary components of $\mathcal{T}$ are bigons.

There is a (possibly collapsed) annulus $\mathcal{B}$ with $C_r$ as its left boundary and $C_l$ as its right boundary. Assume $\mathcal{B}$ is nowhere collapsed. All branches inside $\mathcal{B}$ that are incident to $C_r$ (resp. $C_l$) come into (not out of) $C_r$ (resp. $C_l$). By the annulus construction above, there exists a right-transverse arc $\alpha$ across $\mathcal{B}$, from $C_r$ to $C_l$, and $\alpha$ can be extended across the complement $\mathcal{B}'$ of $\mathcal{B}$ (also an annulus) so as to obtain a right-transverse arc $\alpha$ from $C_r$ back to $C_r$. We can further extend $\alpha$ beyond its endpoint by following $C_r$ \emph{along the orientation of $\mathcal{T}$} till the starting point of $\alpha$. This produces a simple closed curve, still denoted $\alpha$, such that $([\alpha], [C_r])$ forms a $\mathbb{Z}$-basis of the first homology of the torus $\partial M$. By pushing the portion $\alpha \cap C_r$ of $\alpha$ slightly to the right, into the annulus $\mathcal{B}$, we make $\alpha$ completely right-transverse. (If $\mathcal{B}$ is collapsed, then $\mathcal{B}$ is in fact reduced to the oriented cycle $C_r=C_l$ with only incoming branches on both sides, and the construction is essentially the same, crossing only the annulus $\mathcal{B}'$.)

Since $\lambda$ and $-\lambda$ are both right-transverse, we obtain that all slopes of the form $m[\alpha]+n\lambda$ with $m>0$ and $n\in \mathbb{Z}$ are right-transverse. Their negatives are all left-transverse, and this covers all slopes. Lemma \ref{traintracks} is proved.
\end{proof}

We can now finish the proof of Theorem \ref{main}. Recall the point $\theta \in \widetilde{\Theta}$ which minimizes (alphanumerically) the sum of the negative parts of the $\theta_j$ and the number $\nu=|Y'|=\text{Card}\,\{j\in \{1,\dots,3m\}~|~ \theta_j\leq 0\}$. If $\nu>0$, use Definition \ref{arrows} and Claim \ref{extend} to construct the (non-empty) oriented train track $\mathcal{T}$, carried by the graph $G\subset \partial M$, as above. By Lemma \ref{traintracks}, the homotopy class $\mu$ of the meridian of the solid torus $M$ is either carried by $\mathcal{T}$, or transverse to $\mathcal{T}$. But Proposition \ref{nodisks} implies $\mu$ cannot be carried because it bounds a disk. Therefore, $\mu$ is transverse and essentially intersects $\mathcal{T}$ (if $\mu$ were homotopically disjoint from $\mathcal{T}$, it would also be carried). We can therefore find a transverse curve $\gamma$ in the class $\mu$, which always intersects $\mathcal{T}$ with the same sign. Up to reversing direction, $\gamma$ always crosses $\mathcal{T}$ in the same direction as the arrows of $Y'$. Corollary \ref{gentang} says that $\theta+\varepsilon \cdot \Phi (T_\gamma)$ still satisfies (\ref{sigma}). By Observation \ref{recount}, the nonpositive coordinates of $\theta$ (i.e. the $\theta_j$ with $j\in \iota(Y')$) are increased by such a deformation for small positive $\varepsilon$, exactly as in the proof of Claim \ref{extend}. This is contrary to our assumption on $\theta$. Therefore $\nu=0$. Theorem \ref{main} is proved.


\begin{flushright}
Fran\c{c}ois Gu\'eritaud \\
CNRS et Universit\'e de Lille--1 \\
Laboratoire de math\'ematiques Paul--Painlev\'e \\
UMR 8524 du CNRS \\
59655 Villeneuve d'Ascq Cedex, France
\end{flushright}

\end{document}